\documentclass{imsart}

\RequirePackage[OT1]{fontenc}
\RequirePackage{amsthm,amsmath,amssymb, natbib}
\RequirePackage[colorlinks,citecolor=blue,urlcolor=blue]{hyperref}
\usepackage{hypernat}

\startlocaldefs
\numberwithin{equation}{section}
\theoremstyle{plain}

\theoremstyle{remark}

\endlocaldefs

\newtheorem{theorem}{Theorem}[section]

\newtheorem{lemma}[theorem]{Lemma}

\newcommand{\E}{{\mathbb E}}

\newcommand{\R}{{\mathbb R}}
\renewcommand{\P}{{\mathbb P}}

\newcommand{\C}{{\mathcal{C}}}
\newcommand{\dr}{{\mathfrak{d}}}

\newcommand{\F}{{\cal F}}

\newcommand{\Rba}{R_{\rm Bayes}}
\newcommand{\Rnm}{R_{\rm nor}}

\newcounter{rcnt}[section]

\let\hat\widehat
\def\qt#1{\qquad\text{#1}}

\def\argmin{\mathop{\rm argmin}}
\def\argmax{\mathop{\rm argmax}}

\begin{document}

\begin{frontmatter}
\title{A Note on the Approximate Admissibility of Regularized
  Estimators in the Gaussian Sequence Model}
\runtitle{Approximate Admissibility of Regularized
  Estimators}

\begin{aug}
\author{\fnms{Xi} \snm{Chen}\thanksref{a}\ead[label=e1]{xichen@nyu.edu}},
\author{\fnms{Adityanand} \snm{Guntuboyina}\thanksref{b,t1}\ead[label=e2]{aditya@stat.berkeley.edu}}
\and
\author{\fnms{Yuchen} \snm{Zhang}\thanksref{c}
\ead[label=e3]{zhangyuc@cs.stanford.edu}}

\thankstext{t1}{Supported by NSF Grant DMS-1309356}
\runauthor{Chen, Guntuboyina, and Zhang}

\affiliation{New York University\thanksmark{m1} and UC Berkeley\thanksmark{m2} and Stanford University\thanksmark{m3}}

\address[a]{Stern School of Business\\
New York University\\
New York, New York, 10012\\
\printead{e1}\\
}

\address[b]{Department of Statistics\\
University of California, Berkeley\\
Berkeley, CA, 94720\\
\printead{e2}\\
}

\address[c]{Computer Science Department\\
Stanford University\\
Stanford, CA, 94305\\
\printead{e3}\\
}
\end{aug}

\begin{abstract}
We study the problem of estimating an unknown vector $\theta$ from an observation $X$ drawn
according to the normal distribution with mean $\theta$ and identity
covariance matrix under the knowledge that $\theta$ belongs to a known
closed convex set $\Theta$. In this general setting,
\citet{chatterjee2014new} proved that the natural constrained least
squares estimator is ``approximately admissible'' for
every $\Theta$. We extend this result by proving that the same
property holds for all convex penalized estimators as well. Moreover,
we simplify and shorten the original proof considerably. We also
provide explicit upper and lower bounds for the universal constant
underlying the notion of approximate admissibility.
\end{abstract}


\end{frontmatter}

\section{Introduction}
The Gaussian sequence model is a commonly used model for theoretical
investigations in nonparametric and high dimensional statistical
problems. Here one models the data vector $X \in \R^n$ as an
observation having the normal distribution with unknown mean $\theta
\in \R^n$ and identity covariance matrix i.e., $X \sim N(\theta,
I_n)$. Often one assumes some structure on the unknown mean $\theta$
in the form of a convex
constraint. Specifically, it is common to assume that $\theta \in
\Theta$ for some closed convex subset $\Theta$ of $\R^n$. A natural
estimator for $\theta$ under the constraint $\theta \in \Theta$ is the
least squares estimator (LSE) defined as
\begin{equation}\label{lse}
  \widehat{\theta}(X; \Theta)  := \argmin_{\alpha \in \Theta} \frac{1}{2}
  \|X - \alpha \|_2^2
\end{equation}
where $\|\cdot\|_2$ denotes the usual Euclidean norm on $\R^n$. It is
easy to see that many common estimators in nonparametric and
high-dimensional statistics such as shape constrained estimators (see,
for example, \citet{groeneboom2014nonparametric}) and those based on
constrained LASSO (see, for example, \citet{buhlmann2011statistics})
are special cases of the LSE \eqref{lse} for  various choices of
$\Theta$.


In this abstract setting, \citet{chatterjee2014new} asked the
following question: Does the estimator $\widehat{\theta}(X; \Theta)$
satisfy a general optimality property that holds for \textit{every}
closed convex set $\Theta$? This is a non-trivial question; obvious
guesses for the optimality property might be admissibility and
minimaxity but the LSE does not satisfy either of these for every
$\Theta$. Indeed, $\widehat{\theta}(X; \Theta)$ is not minimax (even
up to multiplicative factors that do not depend on the dimension $n$)
when $\Theta := \{\alpha \in \R^n : \sum_{i < n} \alpha_i^2 + n^{-1/2}
\alpha_n^2 \leq 1\}$ as noted by \citet{zhang2013nearly} (a more
elaborate counterexample for minimaxity is given in
\citet{chatterjee2014new}). Also, $\widehat{\theta}(X; \Theta)$ is not
admissible when $\Theta = \R^n$ where the James-Stein estimator
dominates $\hat{\theta}(X) = X$ (see, for example,
\citet{LehmannCasella}).

\citet{chatterjee2014new} answered the general optimality question of
the constrained LSE in the affirmative by proving that $\widehat{\theta}(X; \Theta)$ is
\textit{approximately admissible} over $\Theta$ for every
$\Theta$. The precise statement of \citeauthor{chatterjee2014new}'s theorem is described
below. Let us say that, for a constant $C > 0$, an estimator $\dr(X)$
is \emph{$C$-admissible} over $\Theta$ if for every other estimator
$\widetilde{\dr}(X)$, there exists $\theta \in \Theta$ such
that
\begin{equation}\label{cad}
  C \E_{\theta} \|\dr(X) - \theta \|_2^2 \leq \E_{\theta}
  \|\widetilde{\dr}(X) - \theta\|_2^2.
\end{equation}
In words, the above definition means that for every
estimator $\widetilde{\dr}(X)$, there exists a point $\theta \in
\Theta$ at which the estimator $\dr(X)$ performs as well as the
estimator $\widetilde{\dr}(X)$ up to the multiplicative factor
$C$. Note that the point at which $\dr(X)$ performs better than
$\widetilde{\dr}(X)$ would depend on the estimator
$\widetilde{\dr}(X)$ as well as on the constraint set $\Theta$.
Essentially an estimator $\dr(X)$ being $C$-admissible over $\Theta$
means that it is impossible for any estimator to dominate $\dr(X)$
uniformly over $\Theta$ by more than the multiplicative factor $C$.

\citet{chatterjee2014new} proved that
there exists a \textit{universal} constant $0 < C \leq 1$ such that for every
$n \geq 1$ and closed  convex subset $\Theta \subseteq \R^n$, the LSE
$\widehat{\theta}(X; \Theta)$ is  $C$-admissible for $\Theta$.

\begin{theorem}\label{jee}[\cite{chatterjee2014new}]
There exists a universal constant $0 < C \leq 1$ (independent of $n$
and $\Theta$) such that for every $n \geq 1$ and closed  convex subset
$\Theta \subseteq \R^n$, the least squares estimator
$\widehat{\theta}(X; \Theta)$ is  $C$-admissible over
$\Theta$.
\end{theorem}

Remarkable features of the above theorem are that it is true for every
$\Theta$ and that the constant $C$  does not depend on $n$ or
$\Theta$. We would like to mention here that Theorem \ref{jee} is a
rather difficult result (in Chatterjee's own words, ``from a purely
mathematical point of view, this is the deepest result of this
paper'') and the original proof in \cite{chatterjee2014new} is quite
complex.

Our paper has the following twin goals: (a) we extend Theorem
\ref{jee} (which only involves constrained estimators) to penalized
estimators, which are more commonly used in practice, and (b) 
we simplify considerably the proof of Theorem \ref{jee} given
  in \cite{chatterjee2014new} and  our proof is also much more
  intuitive. To describe our main result, let us first introduce
  penalized  estimators. Given a closed convex set $\Theta \subseteq
  \R^n$ and a real-valued convex function $f$ on $\Theta$, let
\begin{equation}\label{eq:obj}
  \widehat{\theta}(X; \Theta, f) := \argmin_{\alpha \in \Theta}
  \left(\frac{1}{2} \|X - \alpha\|^2 + f(\alpha) \right).
\end{equation}
Strictly speaking $\widehat{\theta}(X; \Theta, f)$ is a least squares
estimator that is both constrained and penalized. We can
of course write it as a pure penalized estimator with the penalty
function $\tilde{f}(x)=f(x)+\mathbb{I}_{\Theta}(x)$, where
$\mathbb{I}_\Theta(x)$ is the indicator function that takes the
    value 0 when $x \in \Theta$ and $+\infty$ otherwise. We choose to
    separate the constraint and penalty as it is more natural for many
    statistical applications. In doing so, note that we have required
    that $f$ is real-valued (i.e., $f$ does not take the value
    $+\infty$) on $\Theta$.


For $\Theta = \R^n$ in \eqref{eq:obj}, we obtain penalized
estimators for which the LASSO is the most common example. For $f
\equiv 0$, we get back the constrained LSEs of \eqref{lse}. There are
examples where one uses both a non-trivial constraint set $\Theta$ and
a non-trivial penalty function $f(\cdot)$; for example, in isotonic
regression, it is common to use
\begin{equation*}
  \Theta := \left\{\alpha \in \R^n : \alpha_1 \leq \dots \leq \alpha_n
  \right\} ~~ \text{ and } ~~ f(\alpha) := \lambda \left(\alpha_n -
    \alpha_1 \right)
\end{equation*}
for some $\lambda \geq 0$. This estimator fits non-decreasing
sequences to the data while constraining the range of the estimator so
as to prevent the spiking effect that the usual isotonic LSE suffers
from; see, for example, \citet{woodroofe1993penalized}.

Because of the presence of the penalty function $f$, it is clear that
the class of estimators given by $\widehat{\theta}(X; \Theta, f)$ is
larger compared to the class given by the LSEs in \eqref{lse}. The
main result of our paper is the following.
\begin{theorem}\label{shiv}
  There exists a universal constant $0 < C \leq 1$ (independent of $n,
  \Theta$ and $f$) such that for every $n \geq 1$, closed convex set
  $\Theta \subseteq \R^n$  and real-valued convex function $f$ on
  $\Theta$, the estimator $\widehat{\theta}(X; \Theta, f)$ is
  $C$-admissible over $\Theta$.
\end{theorem}
The above theorem generalizes Theorem \ref{jee} by showing that all
estimators $\widehat{\theta}(X; \Theta, f)$ have the $C$-admissibility
property over $\Theta$ for a universal constant $C$. In words, this
means that given any estimator $\dr(X)$, there exists a point $\theta
\in \Theta$ at which the estimator $\widehat{\theta}(X; \Theta, f)$
performs as well as the estimator $\dr(X)$ up to the multiplicative
factor $C$. This point $\theta \in \Theta$ would depend on the
estimator $\dr(X)$ as well as on the constraint set $\Theta$ and the
penalty function $f$.

It should be noted here that $C$-admissibility (even $1$-admissibility)
does not by itself guarantee than an estimator is good in a reasonable
sense. This is because unnatural estimators that return a fixed vector
in the parameter space are admissible. However, 
  $C$-admissibility could serve as a minimum requirement for an
   estimator to be reasonable. Our main point here is that   
the estimators $\widehat{\theta}(X; \Theta, f)$ are very natural and
commonly used in many applications. Our result  shows that
all these natural estimators (as $\Theta$ and $f$ varies) satisfy
$C$-admissibility for a universal constant $C$.

\subsection{Connections to the normalized minimax risk}
There is a restatement of Theorem \ref{shiv} that is illuminating and
gives a minimax flavor to Theorem \ref{shiv}. Given $\Theta$ and $f$,
let us define the normalized minimax risk over $\Theta$ by
\begin{equation}\label{nori}
  \Rnm(\Theta; f) := \inf_{\dr} \sup_{\theta \in \Theta}
  \frac{\E_{\theta} \|\dr(X) - \theta\|_2^2}{\E_{\theta}
    \|\hat{\theta}(X; \Theta, f) - \theta \|_2^2} .
\end{equation}
where the infimum is over all estimators $\dr(X)$. Here we use the
conventions $\frac{0}{0} = 1$ and $\frac{a}{0} = +\infty$ for $a >
0$. Note that $\Rnm(\Theta; f)$ is defined just like the usual minimax
risk over the parameter space $\Theta$ except that the risk of every
estimator $\dr(X)$ is rescaled (normalized) by the risk of
$\widehat{\theta}(X; \Theta, f)$. This therefore a reasonable measure
of comparison of arbitrary estimators $\dr(X)$ to our estimator
$\widehat{\theta}(X; \Theta, f)$.

It is clear that $\Rnm(\Theta; f) \leq 1$ as can be seen by bounding
the infimum in \eqref{nori} by the term corresponding to $\dr(X) =
\widehat{\theta}(X; \Theta, f)$. A small value for $\Rnm(\Theta; f)$
means that there exists an estimator $\dr(X)$  and some point $\theta
\in \Theta$ at which the risk of $\dr(X)$ is smaller, by a large
factor, than the risk of $\widehat{\theta}(X; \Theta, f)$.

Let us now define a universal constant $C^*$ by taking the \textit{worst
possible value} of $\Rnm(\Theta; f)$ over all possible values of the
dimension $n$, convex constraint set $\Theta$ and convex penalty
function $f$. Specifically, let
\begin{equation}\label{sta}
  C^* := \inf_{n \geq 1} \inf_{\Theta \in \C_n} \inf_{f \in
    \F(\Theta)} \Rnm(\Theta; f)
\end{equation}
where $\C_n$ denotes the class of all closed convex subsets of $\R^n$
and $\F(\Theta)$ denotes the class of all real-valued convex functions
on $\Theta$. Note first that $C^*$ is a universal constant and, a
priori, it is not clear if $C^*$ is zero or strictly positive.

It is now straightforward to verify that Theorem \ref{shiv} is
equivalent to the statement that $C^*$ \textit{is strictly
  positive}. Another contribution of our paper is to provide explicit
lower and upper bounds for $C^*$. 
\begin{theorem}\label{cista}
  The universal constant $C^*$ satisfies
  \begin{equation}
    \label{cista.eq}
   6.05 \times 10^{-6} \leq C^* \leq \frac{1}{2}.
  \end{equation}
\end{theorem}
The lower bound of $6.05 \times 10^{-6}$ for $C^*$ comes from our
argument for the proof of Theorem \ref{shiv}. It must be noted here
that \citet{chatterjee2014new} does not provide any explicit values
for $C$ in his $C$-admissibility result. Even if the constant $C$ were
tracked down in the proof of \cite{chatterjee2014new}, it appears that
it will be smaller than $6.05 \times 10^{-6}$ by several orders of
magnitude. The improvement of the lower bound also shows the
  advantage of using our new arguments in the proof of admissibility. 

The upper bound of $1/2$ for $C^*$ is a consequence of an explicit
construction of $\Theta$ and $f$ such that $\widehat{\theta}(X; \Theta,
f)$ is uniformly dominated over $\Theta$ by a factor of $2$ by another
estimator. We believe that this example is non-trivial. Please see
Section \ref{eg} for the proof of Theorem \ref{cista}.

The determination of the exact value of the constant $C^*$ is likely
to be a very challenging problem, which is left to be a future work.


\subsection{Proof sketch}
As a summary, the contributions of the paper include: (1) a novel and
intuitive proof of a generalization of a result of \cite{chatterjee2014new}
on $C$-admissibility in Theorem \ref{shiv}; (2) explicit bounds
for the worst possible value of the normalized minimax risk $C^*$ in
Theorem \ref{cista}. In this subsection, we provide an outline of 
our proofs of these results.

Admissibility results are almost always proved via Bayesian arguments
involving priors. Analogous to the notion of 
$C$-admissibility, we can define a notion of $C$-Bayes as follows. For
$C > 0$ and a proper prior $w$ over $\Theta$, we say that an estimator
$\dr(X)$ is $C$-Bayes with respect to $w$ if
\begin{equation}
  \label{cba}
  C \int_{\Theta} \E_{\theta} \|\dr(X) - \theta\|_2^2 w(d\theta) \leq
  \Rba(w) := \inf_{\tilde{\dr}} \int_{\Theta} \E_{\theta}
  \|\tilde{\dr}(X) - \theta\|_2^2 w(d \theta)
\end{equation}
where $\Rba(w)$ is the Bayes risk with respect to $w$ and the infimum
in the definition of $\Rba(w)$ above is over all estimators
$\tilde{\dr}$.   

It is now trivial to see that an estimator $\dr(X)$ is $C$-admissible
over $\Theta$ if it is $C$-Bayes for some proper prior $w$ supported
on $\Theta$. As a result, in order to prove
that $\widehat{\theta}(X; \Theta, f)$ is $C$-admissible over $\Theta$,
it is
sufficient to construct a proper prior $w$ on $\Theta$ such that
$\widehat{\theta}(X; \Theta, f)$ is $C$-Bayes with respect to $w$. We
construct such a prior $w$ by modifying the construction of
\citet{chatterjee2014new}  (which only applied to the LSE)
appropriately. See the proof of Theorem \ref{shiv} for the description
of $w$.

For our prior $w$, in order to prove that $\widehat{\theta}(X; \Theta,
f)$ is $C$-Bayes, we need to
\begin{enumerate}
\item bound $\int_{\Theta} \E_{\theta} \|\widehat{\theta}(X; \Theta,
  f) - \theta\|_2^2 w(d \theta)$ from above, and
\item bound $\Rba(w)$ from below
\end{enumerate}
and make sure that the two bounds differ only by the multiplicative
factor $C$. For the first step above, we need to study the risk
properties of the estimator $\widehat{\theta}(X, \Theta, f)$. This is
done in Section \ref{rsk}. For the second step, we apply a recent
general Bayes risk lower bound from \citet{GuntuBayLow}. The
application of this risk bound shortens the proof considerably. In
contrast, \citet{chatterjee2014new} used a bare hands for lower
bounding the Bayes risk via ``a sequence of relatively complicated
technical steps involving concentration inequalities and second moment
lower bounds''.

For proving Theorem \ref{cista}, we first observe that our proof of
Theorem \ref{shiv} also yields the lower bound of $6.05 \times 10^{-6}$ for the
constant $C^*$. For proving that $C^* \leq 1/2$, we explicitly
construct a convex set $\Theta$ over which the normalized minimax risk
is arbitrarily close to $1/2$ (see Section \ref{eg} for details).

The rest of this paper is structured as follows. In Section \ref{rsk},
we describe some results on the risk of the estimator $\widehat{\theta}(X;
  \Theta, f)$. These can be seen as an extension of the results of
\citet{chatterjee2014new} for penalized estimators. We will also
  discuss the connection of our risk bounds to a recent work by
  \citet{van2015concentration}.
Section \ref{mp} contains the proof of Theorem \ref{shiv} while Section
\ref{eg} contains the proof of Theorem \ref{cista}. Section \ref{prr}
contains the proofs for the risk results of penalized estimators from
Section \ref{rsk}.

\section{Risk Behavior of $\widehat{\theta}(X; \Theta, f)$}\label{rsk}
Throughout this section, we fix a closed convex set $\Theta$ in $\R^n$
and a real-valued convex function $f$ on $\Theta$. The data vector $X$
will be generated according to the normal distribution with mean
$\theta$ and identity covariance. We study the risk of the estimator
$\widehat{\theta}(X; \Theta, f)$. The main risk result is Theorem
\ref{erc} below which will be used in the proof of Theorem \ref{shiv}
to bound the quantity
\begin{equation*}
  \int_{\Theta} \E_{\theta} \|\hat{\theta}(X; \Theta, f) -
  \theta\|_2^2 w(d\theta)
\end{equation*}
for a suitable prior $w$.

The basic fact about the estimator $\widehat{\theta}(X; \Theta, f)$
(proved in Theorem \ref{erc} below) is that the loss
$\|\widehat{\theta}(X; \Theta, f) - \theta\|_2$ is concentrated around
a deterministic quantity $t_{\theta}$ which depends on $\theta$, the
constraint set $\Theta$ and the regularizer $f$. The quantity
$t_{\theta}$ is defined as the maximizer of the function $G_{\theta} :
[0, \infty) \rightarrow \R$ over $[0, \infty)$ where
\begin{equation}\label{gmd}
  G_{\theta}(t) := m_{\theta}(t) - \frac{t^2}{2} \qt{with
    $m_{\theta}(t) := \E_{\theta} \left(\sup_{\alpha \in \Theta : \|\alpha -
      \theta\|_2 \leq t}  \left\{\left<X - \theta, \alpha - \theta \right> -
    f(\alpha) \right\} \right)$}.
\end{equation}
The quantity $m_{\theta}(t)$ can be viewed as an extension of notion
of (localized) Gaussian width with the penalty function $f(\alpha)$ (note that $X
- \theta$ is a standard Gaussian random variable). Indeed, when 
$f \equiv 0$, $m_{\theta}(t)$ is the Gaussian width of the set
$\{\alpha-\theta: \alpha \in \Theta,  \|\alpha - \theta\|_2 \leq t\}$.  
The existence of $t_{\theta}$ as a unique maximizer of $G_{\theta}(t)$
over $t \in [0, \infty)$ is proved in Lemma \ref{ttm} (see the end of
this section). We also note that $t_{\theta}$ depends on the
  choice of the penalty $f$. 

\begin{theorem}\label{erc}
Fix $\theta \in \Theta$ and consider the estimator
$\widehat{\theta}(X; \Theta, f)$ constructed from $X$ generated
according to the model $X \sim N(\theta, I_n)$. Then
\begin{equation} \label{erc.3}
\P \left\{ \|\widehat{\theta}(X; \Theta, f) - \theta \|_2  \geq
  t_{\theta} + \delta \right\} \leq 2 \exp \left(-
  \frac{\delta^4}{32(t_{\theta} +   \delta)^2} \right)
\end{equation}
for every $\delta \geq 0$ and
  \begin{equation}\label{1co}
    \E_{\theta} \|\widehat{\theta}(X; \Theta, f) - \theta \|_2^2 \leq
    t^2_{\theta} + \left(2 \sqrt{84} \right) t_{\theta} \min(\sqrt{t_{\theta}}, 1) +
    84 \min(t_{\theta}, 1).
  \end{equation}
\end{theorem}

When $f \equiv 0$ i.e., when the estimator $\widehat{\theta}(X;
\Theta, f)$ becomes the LSE over $\Theta$, then the above result has
been proved by \citet[Theorem 1.1]{chatterjee2014new}. Therefore,
Theorem \ref{erc} can be seen as an extension of \citet[Theorem
1.1]{chatterjee2014new} to penalized
estimators. \citet{muro2015concentration} also studied concentration
for penalized estimators; however their result (see \citet[Theorem
1]{muro2015concentration}) proves concentration for (in our notation)
the quantity
\begin{equation*}
\tau(\widehat{\theta}(X; \Theta, f)) := \sqrt{\|\widehat{\theta}(X;
  \Theta, f) - \theta\|_2^2 + 2 f(\widehat{\theta}(X; \Theta, f))}.
\end{equation*}
More recently, \citet{van2015concentration} studied
concentration of the loss of empirical risk minimization estimators in
a very general setting. Two of their results are relevant to Theorem
\ref{erc}. In \citet[Theorem 2.1]{van2015concentration}, it is proved that
$\|\widehat{\theta}(X;   \Theta, f) - \theta \|_2$ concentrates around
its expectation $\E_{\theta} \|\widehat{\theta}(X;   \Theta, f) - \theta \|_2$
at a rate that is faster than that given by Theorem
\ref{erc}. However to prove our admissibility result, we require
concentration of $\|\widehat{\theta}(X;   \Theta, f) - \theta \|_2$
around $t_{\theta}$ and not around $\E_{\theta}\|\widehat{\theta}(X;
\Theta, f) - \theta \|_2$. The relation between $t_{\theta}$ and
$\E_{\theta}\|\widehat{\theta}(X; \Theta, f) - \theta \|_2$ is not
completely clear. Another result from \citet{van2015concentration}
that is relevant to us is their Theorem 
4.1. However it also gives concentration for the quantity
$\tau(\widehat{\theta}(X; \Theta, f))$ while we require concentration
for $\|\widehat{\theta}(X;   \Theta, f) - \theta\|_2$. It is also worthwhile to note that
\citet{van2015concentration} also studied concentration in models more
general than Gaussian sequence models.



In addition to Theorem \ref{erc}, we shall require some additional
facts about $t_{\theta}$ and the function $m_{\theta}$. These are
summarized in the following result which also includes a statement on
the existence and uniqueness of $t_{\theta}$. For the case of the LSE
(i.e., when $f \equiv 0$), the facts stated in the lemma below are
observed in \citet{chatterjee2014new} and most of the results in the
following lemma are straightforward extensions of the corresponding
facts in \citet{chatterjee2014new}.

\begin{lemma}\label{ttm}
  Recall the functions $G_{\theta}(\cdot)$
  and $m_{\theta}(\cdot)$ from \eqref{gmd}.
  \begin{enumerate}
  \item For every $\theta \in \Theta$, the function $m_{\theta}$ is
    non-decreasing and concave.
  \item For every $\theta \in \Theta$, the function
    $G_{\theta}(\cdot)$ has a unique maximizer
    $t_{\theta}$ on $[0, \infty)$.
  \item For every $\theta \in \Theta$ and $t \geq 0$, we have
    \begin{equation}\label{inc}
      m_{\theta}(t) \leq m_{\theta}(t_{\theta}) + t_{\theta}(t -
      t_{\theta}).
    \end{equation}
  \item For every $\theta \in \Theta$ and $t \geq 0$, we have
    \begin{equation}\label{scon}
      G_{\theta}(t) - G_{\theta}(t_{\theta}) \leq \frac{-(t -
      t_{\theta})^2}{2}.
    \end{equation}
  \item The risk function $\theta \mapsto
    \E_{\theta}\|\widehat{\theta}(X; \Theta, f) - \theta\|_2^2$ is
    smooth in the following sense: for every $\theta_1, \theta_2 \in
    \Theta$, we have
    \begin{equation}\label{smor}
      \E_{\theta_1} \|\widehat{\theta}(X; \Theta, f) - \theta_1\|_2^2
      \leq 2 \E_{\theta_2} \|\widehat{\theta}(X; \Theta, f) -
      \theta_2\|_2^2 + 8 \|\theta_1 - \theta_2\|^2.
    \end{equation}
   \item The function $\theta \mapsto t_{\theta}$ is smooth in the
     following sense: for every $\theta_1, \theta_2 \in \Theta$, we
     have
\begin{equation}\label{moo}
  \left(t_{\theta_1} - \sqrt{\|\theta_1 - \theta_2\|_2^2 +
  4 t_{\theta_1} \|\theta_1 - \theta_2\|_2} \right)_+ \leq
t_{\theta_2} \leq t_{\theta_1} + \sqrt{\|\theta_1 - \theta_2\|_2^2 +
  4 t_{\theta_1} \|\theta_1 - \theta_2\|_2}.
\end{equation}
   Also for every $\theta_1, \theta_2 \in \Theta$ and $\rho \geq 0$,
   we have
     \begin{equation}\label{smoo}
  \left(1 - \sqrt{\rho^2 + 4 \rho} \right)_+ t_{\theta_1} \leq
  t_{\theta_2} \leq \left(1 + \sqrt{\rho^2 + 4 \rho} \right)
  t_{\theta_1}
     \end{equation}
provided $\|\theta_1 - \theta_2\|_2 \leq \rho t_{\theta_1}$.
  \end{enumerate}
\end{lemma}
With  Theorem \ref{erc} and Lemma \ref{ttm} in place, we first provide
the proof of our main result (Theorem \ref{shiv}) in the next
section. The proofs for Theorem \ref{erc} and Lemma \ref{ttm} will be
relegated to the end of the paper in Section \ref{prr}.

\section{Proof of Theorem \ref{shiv}} \label{mp}
We follow the program outlined in the introduction. For proving that
$\widehat{\theta}(X; \Theta, f)$ is $C$-admissible for a constant $C$,
it is enough to demonstrate the existence of a prior $w$ on $\Theta$
such that $\widehat{\theta}(X; \Theta, f)$ is $C$-Bayes with respect
to $w$. As described in the introduction, a key step for proving that
$\widehat{\theta}(X; \Theta, f)$ is $C$-Bayes involves bounding from
below the Bayes risk $\Rba(w)$ with respect to $w$. For this purpose,
we shall use the following result from \citet[Corollary
4.4]{GuntuBayLow}. This result states that the following inequality
holds for every prior $w$ on $\Theta$:   
\begin{equation}\label{eq:intro_I_bayes_chi}
  \Rba(w) \geq \frac{1}{2} \sup \left\{t > 0 : \sup_{a \in \Theta} w
   \{\theta \in \Theta : \|\theta - a\|_2^2 \leq t \} < \frac{1}{4(1 +
      I)} \right\}
\end{equation}
where $I$ is any nonnegative number satisfying
\begin{equation}\label{chia}
  I \geq \inf_{Q} \int_{\Theta} \chi^2(P_{\theta} \| Q) d w(\theta).
\end{equation}
Here $P_{\theta}$ denotes the $n$-dimensional normal distribution with
mean $\theta$ and identity covariance and the infimum in \eqref{chia}
is over all probability measures $Q$ on $\R^n$. Also $\chi^2(P \| Q)$
denotes the chi-square divergence defined as $\int (p^2/q) d\mu - 1$
where $p$ and $q$ are densities of $P$ and $Q$ respectively with
respect to a common dominating measure $\mu$. 

We are now ready to prove Theorem \ref{shiv}.

\begin{proof}[Proof of Theorem \ref{shiv}]
We break the proof into two separate cases: the case when
$\inf_{\theta \in \Theta}  t_{\theta} $ is strictly smaller than some
constant $b$ and the case when $\inf_{\theta \in \Theta} t_{\theta}$
is larger than $b$. The first case is the easy case where we show that
$\widehat{\theta}(X; \Theta, f)$ is $C$-Bayes with respect to a simple
two-point prior via Le Cam's classical two-point testing
inequality. The second case is harder where we use a more
elaborate prior $w$ together with inequality
\eqref{eq:intro_I_bayes_chi} to lower bound $\Rba(w)$.

\textbf{Easy Case:} Here we assume that $\inf_{\theta \in \Theta}
t_{\theta} \leq b$ (the precise value of the constant $b$ will be
specified later). Choose $\theta^* \in  \Theta$ such that
$t_{\theta^*} \leq b$ (note that $\theta \mapsto t_{\theta}$ is
continuous from \eqref{moo} and that $\Theta$ is closed so that such a
$\theta^*$ exists). Let $\theta_1 \in \Theta$ be any maximizer of
$\|\theta^* - \theta\|_2$ as $\theta$ varies over $\{\theta \in
\Theta: \|\theta - \theta^*\|_2 \leq 1\}$. Let $w$ be the uniform
prior over the two-point set $\{\theta^*, \theta_1\}$. The Bayes risk
with respect to $w$ can be easily bounded by below by Le Cam's
inequality (from \citet{LeCam:73AnnStat}) which gives
\begin{equation*}
  R_{\rm Bayes}(w) \geq \frac{1}{4} \|\theta^* - \theta_1\|_2^2
  \left(1 - \|P_{\theta^*} - P_{\theta_1}\|_{TV} \right)
\end{equation*}
where $\|P_{\theta^*} - P_{\theta_1}\|_{TV}$ denotes the total
variation distance between the probability measures $P_{\theta^*}$ and
$P_{\theta_1}$. Pinsker's inequality (see for example \cite[Lemma
2.5]{Tsybakovbook}) now implies 
\begin{equation*}
  2 \|P_{\theta^*} - P_{\theta_1} \|^2_{TV} \leq D(P_{\theta^*} \|
  P_{\theta_1}) = \frac{1}{2} \|\theta^* - \theta_1\|_2^2 \leq
  \frac{1}{2}
\end{equation*}
and hence
\begin{equation}\label{rba}
  R_{\rm Bayes}(w) \geq \frac{1}{8} \|\theta^* - \theta_1\|_2^2.
\end{equation}
By the definition of $\theta_1$, we have $\|\theta_1-\theta^*\|_2 \leq
1$. We consider the following two cases by the value of
$\|\theta_1-\theta^*\|_2$.
\begin{enumerate}
\item $\|\theta^* - \theta_1\|_2 = 1$: Here inequality \eqref{rba} gives
  $R_{\rm Bayes}(w) \geq 1/8$. Further, by the assumption $t_{\theta^*}
  \leq b$ and inequality \eqref{1co}, we have
  \begin{equation*}
    \E_{\theta^*}
  \|\widehat{\theta}(X) - \theta^*\|_2^2 \leq b^2 + (2 \sqrt{84})
  b^{3/2} + 84 b
  \end{equation*}
Moreover, by inequality \eqref{smor} and \eqref{1co},   we have
  \begin{align*}
\E_{\theta_1} \|\widehat{\theta}(X; \Theta, f) - \theta_1\|_2^2 &\leq 2
\E_{\theta^*} \|\widehat{\theta}(X; \Theta, f)   - \theta^*\|_2^2 + 8
\|\theta^* - \theta_1\|_2^2 \\
&\leq 2 \left(b^2 + (2 \sqrt{84})
  b^{3/2} + 84 b \right) + 8 \|\theta^* - \theta_1\|_2^2 \\
&\leq   2 \left(b^2 + (2 \sqrt{84})
  b^{3/2} + 84 b \right) + 8.
  \end{align*}
Combining the above two inequalities, we deduce that
  \begin{equation*}
\int_{\Theta} \E_{\theta} \|\widehat{\theta}(X; \Theta, f) -
\theta\|_2^2 dw(\theta) \leq \frac{3}{2} \left(b^2 + (2 \sqrt{84})
  b^{3/2} + 84 b \right) + 4.
  \end{equation*}
This inequality together with $R_{\rm Bayes}(w) \geq 1/8$ allow us to obtain
  \begin{equation*}
   \frac{1}{12 \left(b^2 + (2 \sqrt{84})
  b^{3/2} + 84 b \right) + 32} \int_{\Theta} \E_{\theta}
    \|\widehat{\theta}(X; \Theta, f) - \theta\|_2^2 dw(\theta) \leq
    R_{\rm Bayes}(w).
  \end{equation*}
  This means that $\widehat{\theta}(X; \Theta, f)$ is $C$-Bayes with
  respect to $w$ with
  \begin{equation}\label{amj}
    C := \frac{1}{12 \left(b^2 + (2 \sqrt{84}) b^{3/2} + 84 b \right)
      + 32}.
  \end{equation}
\item $\|\theta^* - \theta_1\|_2 < 1$:  In this case, $\gamma :=
  \text{diam}(\Theta) \leq 2$ and $\|\theta^* - \theta_1\|_2 \geq
  \gamma/2$. Inequality \eqref{rba} then gives $R_{\rm Bayes}(w) \geq
  \gamma^2/32$. Also  for every $\theta \in \Theta$, we have
  $\E_{\theta} \|\widehat{\theta}(X; \Theta, f) - \theta\|_2^2  \leq
  \gamma^2$ (because both $\widehat{\theta}(X; \Theta, f)$ and
  $\theta$ are constrained to take values in $\Theta$ whose diameter
  is at most $\gamma$). These two inequalities imply that
  \begin{equation*}
    \frac{1}{32} \int_{\Theta} \E_{\theta} \|\widehat{\theta}(X;
    \Theta, f) - \theta\|_2^2 dw(\theta) \leq  R_{\rm Bayes}(w)
  \end{equation*}
  which means that $\widehat{\theta}(X; \Theta, f)$ is $C$-Bayes with
  respect to $w$ with $C = 1/32$.
\end{enumerate}
Therefore in this easy case, we have proved that $\widehat{\theta}(X;
\Theta, f)$ is $C$-Bayes for some $C$ that is atleast the minimum of
\eqref{amj} and $1/32$.

\textbf{Hard Case:} We now work with the situation when $\inf_{\theta
  \in \Theta} t_{\theta} > b$. We fix a specific $\theta^* \in
\Theta$ and choose $w$ as a  specific prior that is supported on the
set
\begin{equation}\label{uri}
U(\theta^*): = \Theta \cap \{\theta \in \R^n : \|\theta - \theta^*\|_2
\leq \rho  t_{\theta^*}\}
\end{equation}
for some constant $\rho > 0$ (to be specified later) which satisfies
$\rho^2 + 4 \rho < 1$. More precisely, for a fixed small constant
$\eta$, let $\theta^*$ be chosen so that
\begin{equation}\label{cho}
  m_{\theta^*}(\rho t_{\theta^*}) \geq \sup_{\theta \in \Theta}
  m_{\theta}(\rho t_{\theta}) - \eta
\end{equation}
where $m_{\theta}(\cdot)$ is defined in \eqref{gmd}. Let $\Psi : \R^n
\mapsto \Theta$ be any measurable mapping such that $\Psi(z)$ is a
maximizer of $\left< z, \alpha - \theta^* \right> - f(\alpha)$  as
$\alpha$ varies in $U(\theta^*)$. Let $w$ be the prior given by the
distribution of $\Psi(Z)$ for a standard Gaussian vector $Z$ in
$\R^n$.

Now because of inequalities \eqref{1co} and \eqref{smor},  we can
write the following for every $\theta \in U(\theta^*)$:
\begin{align}
  \E_{\theta} \|\widehat{\theta}(X; \Theta, f) - \theta\|_2^2 &\leq 2
  \E_{\theta^*} \|\widehat{\theta}(X; \Theta, f) - \theta^*\|_2^2 + 8
  \|\theta - \theta^*\|_2^2 \nonumber \\
&\leq 2 \left(t_{\theta^*}^2 + (2 \sqrt{84}) t_{\theta^*}^{3/2} + 84
  t_{\theta^*} \right)   + 8 \|\theta^* - \theta\|_2^2 \nonumber \\
&\leq 2 \left(t_{\theta^*}^2 + (2 \sqrt{84}) t_{\theta^*}^{3/2} + 84
  t_{\theta^*} \right)   + 8 \rho^2 t^2_{\theta^*} \nonumber \\
&= (2 + 8 \rho^2) t^2_{\theta^*} + (4 \sqrt{84}) t_{\theta^*}^{3/2} +
168 t_{\theta^*}.  \nonumber
\end{align}
This clearly implies
\begin{align}
  \int_{\Theta} \E_{\theta} \|\widehat{\theta}(X; \Theta, f) -
  \theta\|_2^2 d w(\theta) &\leq (2 + 8 \rho^2) t^2_{\theta^*} + (4
  \sqrt{84}) t_{\theta^*}^{3/2} + 168 t_{\theta^*} \nonumber \\
&\leq t_{\theta^*}^2 \left(2 + 8 \rho^2 + 4 \sqrt{84} b^{-1/2} + 168
  b^{-1} \right) \label{sa2}
\end{align}
where the second inequality above follows from the fact that
$t_{\theta^*} \geq \inf_{\theta \in \Theta} t_{\theta} > b$.

The goal now is to provide a lower bound for $\Rba(w)$. We shall use
inequality \eqref{eq:intro_I_bayes_chi} for this purpose. Because the
prior $w$ is concentrated on the convex set $U(\theta^*)$, we can
replace the supremum over $a \in \Theta$ in
\eqref{eq:intro_I_bayes_chi} by the supremum over $a \in
U(\theta^*)$. This gives the following lower bound for $\Rba(w)$:
\begin{equation}\label{wf}
  \Rba(w) \geq \frac{1}{2} \sup \left\{t > 0 : \sup_{a \in U(\theta^*)} w \{\theta
  \in \Theta:  \|\theta - a\|_2^2  \leq t\} < \frac{1}{4(1 + I)}
\right\}
\end{equation}
where $I$ is any upper bound on $\inf_{Q} \int_{\Theta}
\chi^2(P_{\theta} \| Q) d w(\theta)$ . Here $P_{\theta}$ is the
$n$-dimensional normal distribution with mean zero and identity
covariance and the infimum is over all probability measures $Q$.

To obtain a suitable value for $I$, we use
\begin{align*}
  \inf_{Q} \int_{\Theta} \chi^2(P_{\theta} \| Q) d w(\theta)  &\leq
  \int_{\Theta} \chi^2(P_{\theta} \| P_{\theta^*}) d w(\theta) \\
&= \int_{U(\theta^*)} \chi^2(P_{\theta} \| P_{\theta^*}) d w(\theta) \\
&\leq \sup_{\theta \in U(\theta^*)} \chi^2(P_{\theta} \| P_{\theta^*})
\leq \exp(\rho^2 t^2_{\theta^*}) - 1
\end{align*}
where, in the last inequality, we used the expression
$\chi^2(P_{\theta} \| P_{\theta^*}) = \exp(\|\theta -
\theta^*\|_2^2)-1$ and the fact that $\|\theta - \theta^*\|_2 \leq
\rho t_{\theta^*}$ for all $\theta \in U(\theta^*)$. We can therefore
take $1 + I$ to be $\exp(\rho^2 t^2_{\theta^*})$ in \eqref{wf} which
gives
\begin{equation}\label{neal}
\Rba(w) \geq \frac{1}{2} \sup \left\{t > 0 : \sup_{a \in U(\theta^*)}
  w \{\theta \in \Theta:  \|\theta - a\|_2^2 \leq t\} < \frac{1}{4}
  \exp(- \rho^2 t^2_{\theta^*}) \right\}.
\end{equation}
We shall now bound from above
\begin{equation}\label{ttd}
  \sup_{a \in U(\theta^*)} w \{\theta \in \Theta:  \|\theta - a\|_2^2
  \leq t\} \qt{for $\sqrt{t} = (1 - \beta) \rho \left(1 - \sqrt{\rho^2
        + 4 \rho} \right) t_{\theta^*} $}
\end{equation}
for a constant $\beta \in (0, 1)$. The goal is to show that the above
quantity is smaller than $\exp(- \rho^2 t^2_{\theta^*})/4$.

Because $w$ is defined as the distribution of $\Psi(Z)$ which is a
maximizer of $\left<Z, \alpha - \theta^* \right> - f(\alpha)$ over
$\alpha \in  U(\theta^*)$, the inequality
\begin{equation*}
  w(A) \leq \P \left\{\sup_{\alpha \in A} \left( \left<Z, \alpha - \theta^*
    \right> - f(\alpha) \right) \geq \sup_{\alpha \in U(\theta^*)}
  \left( \left<Z, \alpha -
      \theta^* \right> - f(\alpha) \right) \right\}
\end{equation*}
holds for every measureable subset $A$ of $\R^n$. Therefore for every
$a \in U(\theta^*)$, the prior probability $w \{\theta \in \Theta :
\|\theta - a\|^2_2 \leq t\}$ is bounded from above by
\begin{equation*}
  \P \left\{\sup_{\theta \in \Theta : \|\theta - a\|^2_2 \leq t}
   \left( \left<Z, \theta - \theta^* \right> - f(\theta) \right)
   \geq \sup_{\theta \in \Theta : \|\theta - \theta^*\| \leq \rho
     t_{\theta^*}} \left( \left<Z, \theta -
      \theta^* \right> - f(\theta) \right) \right\} .
\end{equation*}
The above probability can be exactly written as $\P \{M_2 + M_3 \geq
M_1\}$ where
\begin{equation*}
  M_1 := \sup_{\theta \in \Theta : \|\theta - \theta^*\| \leq \rho
    t_{\theta^*}} \left( \left<Z, \theta - \theta^* \right> -
    f(\theta) \right), ~~~ M_2 :=
  \sup_{\theta \in \Theta : \|\theta - a\|^2_2 \leq t} \left( \left<Z, \theta
    - a  \right> - f(\theta) \right)
\end{equation*}
and $M_3 := \left<Z, a - \theta^* \right>$. Now if $\gamma \geq 0$ is
such that $\E M_1 - \E M_2 \geq \gamma$, then we can write:
\begin{align*}
  \P \{M_2 + M_3 \geq M_1\} &= \P \left\{M_2 - \E M_2 + M_3 + \E M_1 -
  M_1 \geq \E M_1 - \E M_2 \right\} \\
&\leq \P \left\{M_2 - \E M_2 + M_3 + \E M_1 -
  M_1 \geq \gamma \right\} \\
&\leq \P \left\{M_2 - \E M_2 \geq \frac{\gamma}{3} \right\} + \P
\left\{M_3 \geq \frac{\gamma}{3} \right\} + \P \left\{M_1 - \E M_1
  \leq \frac{-\gamma}{3} \right\} \\
&\leq \exp \left(\frac{-\gamma^2}{18 t} \right) + \exp
\left(\frac{-\gamma^2}{18 \|a - \theta^*\|_2^2} \right) + \exp
\left(\frac{-\gamma^2}{18 \rho^2 t^2_{\theta^*}} \right) .
\end{align*}
where the last inequality follows by standard Gaussian concentration
and the observation that (a) $M_2$, as a function of $Z$, is Lipschitz
with Lipschitz constant $\sqrt{t}$, (b) $M_3$, as a function of $Z$,
is Lipschitz with Lipschitz constant $\|a - \theta^*\|_2$, and (c)
$M_1$, as a function of $Z$, is Lipschitz with Lipschitz constant
$\rho t_{\theta^*}$.

We now use the fact that for every $a \in U(\theta^*)$, the inequality
$\|a - \theta^* \|_2 \leq \rho t_{\theta^*}$ holds to deduce that
\begin{align}
w \left\{\theta \in \Theta : \|\theta -
  a\|_2^2 \le t \right\}  &\leq
  \exp \left(\frac{-\gamma^2}{18 t} \right) + 2 \exp
  \left(\frac{-\gamma^2}{18 \rho^2 t^2_{\theta^*}} \right) \nonumber  \\
&\leq 3 \exp \left(\frac{-\gamma^2}{18 \rho^2 t^2_{\theta^*}} \right) ~~
\text{as}~ t \leq \rho^2 t^2_{\theta^*} \text{(see \eqref{ttd})} \label{koy}
\end{align}
Here $\gamma$ is any nonnegative lower bound on $\E M_1 - \E M_2$. To
obtain a suitable value of $\gamma$, we argue as follows. Observe
first that $\E M_1 = m_{\theta^*}(\rho t_{\theta^*})$ and $\E M_2 =
m_a(\sqrt{t})$ where $m$ is
defined in \eqref{gmd}. Because $\theta^*$ is chosen so that inequality
\eqref{cho} is satisfied, we have $\E M_1 =  m_{\theta^*}(\rho
t_{\theta^*}) \geq m_a(\rho t_a) - \eta$. Thus
\begin{equation*}
  \E M_1 - \E M_2 \geq m_a(\rho t_a) - m_a(\sqrt{t}) - \eta.
\end{equation*}
We now use inequality \eqref{smoo} which states that
\begin{equation}\label{ny}
  t_a \geq \left(1 - \sqrt{\rho^2 + 4 \rho} \right) t_{\theta^*}
  \qt{for every $a \in U(\theta^*)$}.
\end{equation}
Now from the expression for $t$ given in \eqref{ttd} and inequality
\eqref{ny} above, it is clear that $\sqrt{t} \leq \rho t_a$ for every
$a \in U(\theta^*)$. Therefore using concavity of $m_a(\cdot)$ (proved
in Lemma \ref{ttm}) and inequality \eqref{inc}, we deduce
\begin{align*}
  \E M_1 - \E M_2 &\geq m_a(\rho t_a) - m_a(\sqrt{t}) - \eta \\
&\geq m_a(t_a) - m_a(t_a - \rho t_a + \sqrt{t}) - \eta \qt{by concavity
  of $m_a(\cdot)$} \\
&\geq t_a \left(\rho t_a - \sqrt{t} \right) - \eta \qt{by inequality
  \eqref{inc}} \\
&\geq \rho \beta t_{\theta^*}^2 \left(1 - \sqrt{\rho^2 + 4\rho}
\right)^2 - \eta  \qt{by inequality \eqref{ny} and the expression for $t$}.
\end{align*}
We therefore take $\gamma$ to be
\begin{equation*}
  \gamma = \rho \beta t_{\theta^*}^2 \left(1 - \sqrt{\rho^2 + 4 \rho}
  \right)^2 - \eta.
\end{equation*}
Inequality \eqref{koy} then gives
\begin{equation*}
\sup_{a \in U(\theta^*)}  w \left\{ \theta \in \Theta : \|\theta -
  a\|_2^2 \leq t \right\} \leq 3 \exp \left(\frac{-\left\{\rho \beta
    t_{\theta^*} (1 - \sqrt{\rho^2 +  4 \rho})^2 - \eta \right\}^2}{18
    \rho^2 t^2_{\theta^*}} \right) .
\end{equation*}
By a straightforward computation, it can be seen that the right hand
side above is strictly smaller than $\frac{1}{4}\exp(-\rho^2
t_{\theta^*}^2)$ if and only if
\begin{equation}\label{sufi1}
  \frac{1}{18 \rho^2 t_{\theta^*}^2} \left(\rho \beta t_{\theta^*}^2
    \left(1 - \sqrt{\rho^2 + 4 \rho} \right)^2 - \eta \right)^2 - \rho^2
  t_{\theta^*}^2 > \log 12
\end{equation}
Now, as a result of the following inequality (note that we are working
under the condition $\inf_{\theta \in \Theta} t_{\theta} > b$ which
implies that $t_{\theta^*} > b$):
\begin{align*}
  \frac{1}{18 \rho^2 t_{\theta^*}^2} \left(\rho \beta t_{\theta^*}^2
    \left(1 - \sqrt{\rho^2 + 4 \rho} \right)^2 - \eta \right)^2 &\geq
  \frac{1}{18 \rho^2 t_{\theta^*}^2} \left(\rho \beta t_{\theta^*}^2
    \left(1 - \sqrt{\rho^2 + 4 \rho} \right)^2 - \eta
    \frac{t_{\theta^*}^2}{b^2} \right)^2  \\
&= \frac{t^2_{\theta^*}}{18 \rho^2} \left(\rho \beta
    \left(1 - \sqrt{\rho^2 + 4 \rho} \right)^2 - \eta b^{-2}
  \right)^2,
\end{align*}
a sufficient condition for \eqref{sufi1} is
\begin{equation}\label{sufi2}
  t_{\theta^*}^2 > (\log 12) \left(\frac{1}{18 \rho^2}\left(\rho \beta (1 -
      \sqrt{\rho^2 + 4 \rho})^2 - \eta b^{-2} \right)^2 - \rho^2
  \right)^{-1}
\end{equation}
We now make the choices:
\begin{equation}\label{och}
  \rho = 0.0295 ~~~~~~~ \beta = 0.42 ~~~~~~~ \eta = 10^{-20} ~~~~~~~ b = 51.53.
\end{equation}
With these, the right hand side of \eqref{sufi2} can be calculated to
be strictly smaller than $b^2$ so that the condition \eqref{sufi2}
holds because $t_{\theta^*} > b$. Therefore we deduce from inequality
\eqref{neal} that
\begin{equation*}
  \Rba(w) \geq \frac{t}{2} = \frac{\rho^2}{2}(1 - \beta)^2 \left(1 -
    \sqrt{\rho^2 + 4 \rho} \right)^2 t_{\theta^*}^2.
\end{equation*}
Combining the above inequality with \eqref{sa2}, we obtain
\begin{equation*}
\frac{\rho^2 (1 - \beta)^2 (1 - \sqrt{\rho^2 + 4 \rho})^2}{2(2 + 8
  \rho^2 + 4 \sqrt{84} b^{-1/2} + 168 b^{-1})} \int_{\Theta}
\E_{\theta}
\|\widehat{\theta}(X; \Theta, f) - \theta\|_2^2 dw(\theta) \leq
\Rba(w)
\end{equation*}
The constant above (for our choice of $\rho, \beta$ and $b$ in
\eqref{och}) is at least $6.05 \times 10^{-6}$. This means therefore
that $\widehat{\theta}(X; \Theta, f)$ is $C$-Bayes with respect to $w$
with $C \geq 6.05 \times 10^{-6}$ in the case when $\inf_{\theta \in
  \Theta} t_{\theta} > b = 51.53$.

It is also easy to check that for $b = 51.53$, the constant in
\eqref{amj} is also at least $6.05 \times 10^{-6}$. Therefore in every
case, we have proved the existence of a prior $w$ such
that $\widehat{\theta}(X; \Theta, f)$ is $C$-Bayes with respect to $w$
for $C \geq 6.05 \times 10^{-6}$. This means that $\widehat{\theta}(X;
\Theta, f)$ is $C$-admissible for some constant $C \geq 6.05 \times
10^{-6}$. This completes the proof of Theorem \ref{shiv}.
\end{proof}

\section{Proof of Theorem \ref{cista}} \label{eg}
Theorem \ref{shiv} shows that for every $n \geq 1$, $\Theta \in \C_n$
and $f \in \F(\Theta)$, the estimator $\widehat{\theta}(X; \Theta, f)$
is $C$-admissible over $\Theta$ for some $C \geq 6.05 \times
10^{-6}$. This immediately implies that $C^* \geq 6.05 \times
10^{-6}$. We therefore only need to prove that $C^* \leq 1/2$.

For this it is enough to show that for every $\epsilon > 0$, there
exists a closed convex set $\Theta \subseteq
\R$ such that the LSE $\widehat{\theta}(X; \Theta)$ satisfies
\begin{equation}\label{tmo}
  \inf_{\dr} \sup_{\theta \in \Theta} \frac{\E_{\theta} \|\dr(X) -
    \theta\|_2^2}{\E_{\theta} \|\widehat{\theta}(X; \Theta) -
    \theta\|_2^2}  \leq \frac{1}{2} + \epsilon.
\end{equation}
Note that $\widehat{\theta}(X; \Theta)$ is a special case of
$\widehat{\theta}(X; \Theta, f)$ for $f \equiv 0$.  

Let $\Theta := [-a, a]$ for some $a > 0$ (to be specified later). It
is then clear that
$$
\widehat{\theta}(X; \Theta) = \left\{ \begin{array}{rl}
 X &\mbox{ if $-a \leq X \le a$} \\
 a &\mbox{ if $X > a$} \\
-a & \mbox{ if $X < -a$}
       \end{array} \right.
$$
Note now that
\begin{equation*}
  \inf_{\theta \in \Theta} \P_{\theta} \{X > a\} = P_{-a} \{X > a\} = 1 -
  \Phi(2a).
\end{equation*}
and similarly
\begin{equation*}
  \inf_{\theta \in \Theta} \P_{\theta} \{X < -a\} = 1 - \Phi(2a).
\end{equation*}
Therefore for every $\theta \in \Theta = [-a, a]$, we have
\begin{align*}
  \E_{\theta} \|\widehat{\theta}(X; \Theta) - \theta\|_2^2 &\geq
  \E_{\theta} \left( \|\widehat{\theta}(X; \Theta) -
  \theta\|_2^2 \{X > a\} \right) + \E_{\theta} \left( \|\widehat{\theta}(X; \Theta) -
  \theta\|_2^2 \{X < -a\} \right) \\
&= (\theta - a)^2 \P_{\theta} \{X > a\} + (\theta + a)^2 \P_{\theta} \{X
< -a\} \\
&\geq (\theta - a)^2 \P_{-a} \{X > a\} + (\theta + a)^2 \P_{a} \{X
< -a\} \\
&\geq \left(1 - \Phi(2a)\right) \left[(\theta - a)^2 + (\theta + a)^2
\right] = 2 \left(1 - \Phi(2a)\right) \left(\theta^2 + a^2 \right)
\end{align*}
where in the third line above, we used the fact that $\inf_{\theta \in
[-a, a]}\P_{\theta} \{X > a\} = \P_{-a} \{X > a\}$ and $\inf_{\theta \in
[-a, a]}\P_{\theta} \{X < -a\} = \P_{a} \{X > a\}$

As a result, we deduce that
\begin{align*}
\inf_{\dr} \sup_{\theta \in \Theta} \frac{\E_{\theta} \|\dr(X) -
    \theta\|_2^2}{\E_{\theta} \|\widehat{\theta}(X; \Theta) -
    \theta\|_2^2}  &\leq \sup_{\theta \in \Theta}
  \frac{\theta^2}{\E_{\theta} \|\widehat{\theta}(X; \Theta) -
    \theta\|_2^2} \\
 &\leq \frac{1}{2(1 - \Phi(2a))} \sup_{\theta \in
    \Theta} \frac{\theta^2}{\theta^2 + a^2} = \frac{1}{4(1 -
    \Phi(2a))}
\end{align*}
where in the first inequality above, we bounded the infimum over all
estimators $\dr$ by the simple estimator $\dr(X) \equiv 0$.  Note now
that by taking the limit $a \downarrow 0$, the right hand side above
goes to $1/2$. Therefore it is possible to choose $a$ small enough
depending on $\epsilon$ to ensure \eqref{tmo}. This proves $C^* \leq
1/2$ and completes the proof of Theorem \ref{cista}.

\section{Proofs of Theorem \ref{erc} and Lemma \ref{ttm}} \label{prr}
In this section, we shall provide proofs for Theorem \ref{erc} and
Lemma \ref{ttm}. We first give the proof of Lemma \ref{ttm} below
before proceeding to the proof of Theorem \ref{erc}. This is because
parts of Lemma \ref{ttm} will be useful for proving Theorem
\ref{erc}.

\begin{proof}[Proof of Lemma \ref{ttm} ]
  \textbf{1.} Fix $\theta \in \Theta$. We need to prove that
  $m_{\theta}(\cdot)$ is a non-decreasing
  and concave function. It is trivial to see that $m_{\theta}$ is a
  non-decreasing function because the sets $\{\alpha \in \Theta :
  \|\alpha - \theta\|_2 \leq t\}$ are increasing in $t$. To prove
  concavity of $m_{\theta}$, observe that it is enough to show that
  \begin{equation*}
    H_{\theta}(t) := \sup_{\alpha \in \Theta: \|\alpha - \theta\|_2 \leq t}
    \left(\left<X - \theta, \alpha - \theta \right> - f(\alpha) \right)
  \end{equation*}
  is concave for every $z \in \R^n$.  This is because $m_{\theta}(t) =
  \E H_{\theta}(t)$. To prove concavity of $H_{\theta}$, let us fix $0
  \leq t_1 <   t_2 < \infty$, $0 < u< 1$ and $t = (1 - u) t_1 + u
  t_2$. For every $\eta > 0$, by the definition of $H_{\theta}(t)$,
  for each $i
  = 1, 2$, there exists $\theta_i \in \Theta$ with $\|\theta_i -
  \theta\| \leq t_i$
  such that
  \begin{equation*}
    H_{\theta}(t_i) \leq \left<X - \theta, \theta_i - \theta \right> -
    f(\theta_i) + \eta.
  \end{equation*}
Now with $\alpha := (1 -u) \theta_1 + u \theta_2$, it is easy to see
that $\|\alpha - \theta\|_2 \leq t, \alpha \in \Theta$ and, by
convexity of $f$, that
\begin{align*}
  H_{\theta}(t) &\geq  \left<X - \theta, \alpha - \theta \right> - f(\alpha) \\
&\geq (1 - u) \left[ \left<X - \theta, \theta_1 - \theta \right> -
  f(\theta_1)\right]  + u \left[ \left<X - \theta, \theta_2 - \theta \right> -
  f(\theta_2)\right]  \\
&\geq (1 - u) H_{\theta}(t_1) + u H_{\theta}(t_2) - \eta.
\end{align*}
Because $\eta > 0$ is arbitrary, this proves concavity of $H_{\theta}$
on $[0, \infty)$ which implies concavity of $m_{\theta}$ on $[0,
\infty)$.

\textbf{2.} The concavity of $m_{\theta}$ implies that $G_{\theta}(t)
:= m_{\theta}(t) - t^2/2$ is strictly concave on $[0,
\infty)$. Moreover 
\begin{align*}
  m_{\theta}(t) &= \E_{\theta} \left(\sup_{\alpha \in \Theta : \|\alpha -
      \theta\|_2 \leq t}  \left\{\left<X - \theta, \alpha - \theta \right> -
    f(\alpha) \right\} \right) \\
&\leq \E_{\theta} \left(\sup_{\alpha \in \Theta : \|\alpha -
      \theta\|_2 \leq t}  \left<X - \theta, \alpha - \theta \right>
  \right) - \inf_{\alpha \in \Theta} f(\alpha) \\
&\leq \|X - \theta\|_2 t - \inf_{\alpha \in \Theta} f(\alpha)
\end{align*}
where, in the last inequality above, we used the Cauchy-Schwarz
inequality. As a result, $G_{\theta}(t) = m_{\theta}(t) - t^2/2$
converges to $-\infty$ as $t \uparrow +\infty$. This, together with
strict concavity of $G_{\theta}(\cdot)$ on $[0, \infty)$, immediately
imply that $G_{\theta}$ has a unique maximizer $t_{\theta}$ over $[0,
\infty)$.  

\textbf{3.} Fix $\theta \in \Theta$ and $t \ge 0$. Then the inequality
$G_{\theta}(t_{\theta}) \geq G_{\theta}(t)$ holds because $t_{\theta}$
maximizes $G_{\theta}(\cdot)$. This inequality is equivalent to
\begin{equation*}
  m_{\theta}(t) \leq m_{\theta}(t_{\theta}) + \frac{t^2 -
    t^2_{\theta}}{2}.
\end{equation*}
Applying this inequality to $(1 - u) t_{\theta} + u t$ instead of $t$
for a fixed $u \in (0, 1)$, we obtain
\begin{equation*}
  m_{\theta}((1 - u) t_{\theta} + u t) \leq m_{\theta}(t_{\theta}) +
  \frac{(-2u + u^2)t^2_{\theta} + u^2 t^2 + 2u(1 - u) t~t_{\theta}}{2}.
\end{equation*}
Using $m_{\theta}((1 -u) t_{\theta} + u t) \geq (1-
u)m_{\theta}(t_{\theta}) + u~ m_{\theta}(t)$ on the right hand side
above, we get
\begin{equation*}
  u~ m_{\theta}(t) \leq u ~ m_{\theta}(t_{\theta}) +   \frac{(-2u +
    u^2)t^2_{\theta} + u^2 t^2 + 2u(1 - u) t ~t_{\theta}}{2}. .
\end{equation*}
Dividing both sides of the above inequality by $u$ and then letting $u
\rightarrow 0$, we obtain
\begin{equation*}
  m_{\theta}(t_{\theta}) - m_{\theta}(t) \geq t^2_{\theta} - t~ t_{\theta}
\end{equation*}
which proves \eqref{inc}.

\textbf{4.} The expression $G_{\theta}(t) := m_{\theta}(t) - t^2/2$
implies that inequality \eqref{scon} is equivalent to
\eqref{inc}. Therefore, inequality \eqref{scon} follows from the
previous part.

\textbf{5.} To prove \eqref{smor}, we first observe that the map $X
\mapsto \widehat{\theta}(X; \Theta, f)$ is $1$-Lipschitz (for a proof
of this standard fact, see e.g., \citet[Proof of Theorem
2.1]{van2015concentration}). As a result, for every $\theta_1,
\theta_2 \in \Theta$, we can write (for simplicity below, we write
$\widehat{\theta}(X)$ for $\widehat{\theta}(X; \Theta, f)$)
\begin{align*}
  \|\widehat{\theta}(X) - \theta_1\|_2 &\leq
  \|\widehat{\theta}(X) -
  \widehat{\theta}(X-\theta_1+\theta_2)\|_2  +
  \|\widehat{\theta}(X-\theta_1+\theta_2) - \theta_2\|_2 +
  \|\theta_1 - \theta_2 \|_2 \\
&\leq  \|\widehat{\theta}(X-\theta_1+\theta_2) - \theta_2\|_2 +
2  \|\theta_1 - \theta_2 \|_2
\end{align*}
where the first inequality is due to the triangle inequality while the
second inequality is because $\|\widehat{\theta}(X) -
\widehat{\theta}(X - \theta_1 + \theta_2) \|_2 \leq \|\theta_1 -
\theta_2\|_2$ by the $1$-Lipschitz property of $X \mapsto
\widehat{\theta}(X)$. By squaring both sides of the above displayed
inequality and using $(a + b)^2 \leq 2 a^2 + 2 b^2$, we get
\begin{equation*}
  \|\widehat{\theta}(X) - \theta_1\|_2^2 \leq
 2 \|\widehat{\theta}(X-\theta_1+\theta_2) - \theta_2\|^2_2 + 8
  \|\theta_1 - \theta_2 \|^2_2 .
\end{equation*}
Taking expectations with respect to $X \sim N(\theta_1, I_n)$ on both
sides and using the fact that $X - \theta_1 +\theta_2 \sim N(\theta_2,
I_n)$, we obtain the required inequality \eqref{smor}.

\textbf{6. } Fix $\theta_1, \theta_2 \in \Theta$ and observe that, for
every $t \geq 0$, we have
\begin{align*}
  m_{\theta_2}(t) &= \E_{\theta_2} \left(\sup_{\alpha \in \Theta : \|\alpha -
      \theta_2\|_2 \leq t}  \left\{\left<X - \theta_2, \alpha - \theta_2 \right> -
    f(\alpha) \right\} \right) \\
&= \E_{\theta_2} \left(\sup_{\alpha \in \Theta : \|\alpha -
      \theta_2\|_2 \leq t}  \left\{\left<X - \theta_2, \alpha \right> -
    f(\alpha) \right\} \right) \\
&\leq \E_{\theta_2} \left(\sup_{\alpha \in \Theta : \|\alpha -
      \theta_1\|_2 \leq t + \|\theta_1 - \theta_2\|_2}  \left\{\left<X -
        \theta_2, \alpha \right> - f(\alpha) \right\} \right) \\
&= \E_{\theta_1} \left(\sup_{\alpha \in \Theta : \|\alpha -
      \theta_1\|_2 \leq t + \|\theta_1 - \theta_2\|_2}  \left\{\left<X -
        \theta_1, \alpha - \theta_1 \right> - f(\alpha) \right\}
  \right)  = m_{\theta_1}(t + \|\theta_1 - \theta_2\|_2).
\end{align*}
Switching the roles of $\theta_1$ and $\theta_2$, we obtain
\begin{equation*}
  m_{\theta_1}(t) \leq m_{\theta_2}(t + \|\theta_1 - \theta_2\|_2).
\end{equation*}
Combining the above two inequalities, we deduce that
\begin{equation*}
  m_{\theta_1}(t - \|\theta_1 - \theta_2\|_2) \leq m_{\theta_2}(t)
  \leq m_{\theta_1}(t + \|\theta_1 - \theta_2\|_2)
\end{equation*}
provided $\|\theta_1 - \theta_2\|_2 \leq t$. Using $G_{\theta_2}(t) =
m_{\theta_2}(t) - t^2/2$, we further deduce that
\begin{equation*}
  m_{\theta_1}(t - \|\theta_1 - \theta_2\|_2) - \frac{t^2}{2} \leq G_{\theta_2}(t)
  \leq m_{\theta_1}(t + \|\theta_1 - \theta_2\|_2) - \frac{t^2}{2}
\end{equation*}
Therefore for every $t \geq 0$, we have
\begin{align*}
  G_{\theta_2}(t_{\theta_1} + \|\theta_1 - \theta_2\|_2)  -
  G_{\theta_2}(t) &\geq \left\{ m_{\theta_1}(t_{\theta_1}) - \frac{1}{2}
  \left(t_{\theta_1} + \|\theta_1 - \theta_2\|_2 \right)^2 \right\} -
  \left\{m_{\theta_1}(t + \|\theta_1 - \theta_2\|_2) - \frac{t^2}{2}
  \right\} \\
&= \left\{m_{\theta_1}(t_{\theta_1}) - m_{\theta_1}(t + \|\theta_1 -
  \theta_2\|_2) \right\} - \frac{1}{2}\left\{\left(t_{\theta_1} +
    \|\theta_1 - \theta_2\|_2 \right)^2 - t^2 \right\} \\
&\geq t_{\theta_1} \left(t_{\theta_1} - t - \|\theta_1 - \theta_2\|_2
\right) - \frac{1}{2}\left\{\left(t_{\theta_1} +
    \|\theta_1 - \theta_2\|_2 \right)^2 - t^2 \right\}
\end{align*}
where we used \eqref{inc} for the last inequality above. Simplifying
the right hand side above, we deduce that
\begin{equation}\label{tpa}
  G_{\theta_2}(t_{\theta_1} + \|\theta_1 - \theta_2\|_2)  -
  G_{\theta_2}(t) \geq \frac{1}{2} \left(t - t_{\theta_1} \right)^2 -
  \frac{1}{2} \|\theta_1 - \theta_2\|_2^2 - 2 t_{\theta_1} \|\theta_1
  -   \theta_2\|_2.
\end{equation}
From this expression, it follows that
\begin{equation}\label{tpa.1}
  G_{\theta_2}(t_{\text{low}}) \leq G_{\theta_2}(t_{\theta_1} + \|\theta_1 -
  \theta_2\|_2) \leq G_{\theta_2}(t_{\text{up}})
\end{equation}
for
\begin{equation*}
  t_{\text{low}} := t_{\theta_1} - \sqrt{\|\theta_1 - \theta_2\|_2^2 +
  4 t_{\theta_1} \|\theta_1 - \theta_2\|_2} ~~ \text{ and }
t_{\text{up}} := t_{\theta_1} + \sqrt{\|\theta_1 - \theta_2\|_2^2 +
  4 t_{\theta_1} \|\theta_1 - \theta_2\|_2}
\end{equation*}
as long as $t_{\text{low}} \geq 0$. Inequality \eqref{tpa.1},
together with the fact that $G_{\theta_2}(\cdot)$ is strictly concave
on $[0, \infty)$ (this follows from
concavity of $m_{\theta}(\cdot)$ and the observation that $t \mapsto
-t^2/2$ is strictly concave) implies that $t_{\text{low}} \leq
t_{\theta_2} \leq t_{\text{up}}$  as long as $t_{\text{low}} \geq
0$. Because $t_{\theta_2}$ is always nonnegative, we deduce therefore
that
\begin{equation*}
  \left(t_{\theta_1} - \sqrt{\|\theta_1 - \theta_2\|_2^2 +
  4 t_{\theta_1} \|\theta_1 - \theta_2\|_2} \right)_+ \leq
t_{\theta_2} \leq t_{\theta_1} + \sqrt{\|\theta_1 - \theta_2\|_2^2 +
  4 t_{\theta_1} \|\theta_1 - \theta_2\|_2}
\end{equation*}
for every $\theta_1, \theta_2 \in \Theta$ which proves
\eqref{moo}. It is easy to see that inequality \eqref{smoo} is a
simple consequence of \eqref{moo}.
\end{proof}

\begin{proof}[Proof of Theorem \ref{erc}]
   Let $Z := X - \theta \sim N(0, I_n)$. By the calculation
   \begin{equation*}
     \|X - \alpha\|_2^2 = \|Z + \theta - \alpha\|_2^2 = \|Z\|_2^2 - 2
     \left<Z, \alpha - \theta \right> + \|\alpha - \theta\|_2^2,
   \end{equation*}
   it follows that
   \begin{equation}\label{a1}
     \widehat{\theta}(X; \Theta, f) = \argmax_{\alpha \in \Theta}
     \left\{\left<Z, \alpha - \theta \right> - f(\alpha) - \frac{1}{2}
       \|\alpha - \theta\|_2^2 \right\}
   \end{equation}
   For convenience, we shall write $\hat{\theta}$ for
   $\widehat{\theta}(X; \Theta, f)$ in the rest of the proof.

   From \eqref{a1}, one can see as follows that $\|\hat{\theta} -
   \theta\|_2$ maximizes the function $\tilde{A}_{\theta}(t)$ over $t
   \geq 0$ where
   \begin{equation*}
     \tilde{A}_{\theta}(t) := \sup_{\theta \in \Theta : \|\alpha -
       \theta\|_2 = t} \left\{\left<Z, \alpha - \theta \right> -
       f(\alpha) \right\} - \frac{t^2}{2}.
   \end{equation*}
   To see this, just observe that
   \begin{align*}
     \tilde{A}_{\theta}(\|\hat{\theta} - \theta\|_2) &= \sup_{\theta
       \in \Theta : \|\alpha - \theta\|_2 = \|\hat{\theta} -
       \theta\|_2} \left\{\left<Z, \alpha - \theta \right> - f(\alpha)
     \right\} - \frac{\|\hat{\theta} - \theta\|_2^2}{2} \\
&\geq \left<Z, \hat{\theta} - \theta \right> - f(\hat{\theta}) -
  \frac{\|\hat{\theta} - \theta\|_2^2}{2}  \\
&= \sup_{\alpha \in \Theta} \left(\left<Z, \alpha - \theta \right> - f(\alpha) -
  \frac{\|\alpha - \theta\|_2^2}{2}  \right) \\
&\geq \sup_{\alpha \in \Theta : \|\alpha - \theta\|_2 = t}
\left(\left<Z, \alpha - \theta \right> - f(\alpha) - \frac{\|\alpha -
    \theta\|_2^2}{2}  \right) \qt{for every $t \geq 0$} \\
&= \tilde{A}_{\theta}(t).
   \end{align*}
Using this fact, we shall now argue that $\|\hat{\theta} - \theta\|_2$
also maximizes the function $A_{\theta}(t)$ over $t \geq 0$ where
\begin{equation*}
  A_{\theta}(t) := \sup_{\theta \in \Theta : \|\alpha -
       \theta\|_2 \leq t} \left\{\left<Z, \alpha - \theta \right> -
       f(\alpha) \right\} - \frac{t^2}{2}.
\end{equation*}
Note that the difference between $\tilde{A}_{\theta}(t)$ and
$A_{\theta}(t)$ is that the supremum is taken over $\|\alpha -
\theta\|_2 = t$ in $\tilde{A}_{\theta}(t)$ while it is over the larger
set $\|\alpha - \theta\|_2 \leq t$ in $A_{\theta}(t)$. This in
particular means that $\tilde{A}_{\theta}(t) \leq A_{\theta}(t)$ for
every $t \geq 0$. To see that $\|\hat{\theta} - \theta\|_2$ maximizes
$A_{\theta}(t)$, fix $t \geq 0$ and $\eta > 0$. By definition of
$A_{\theta}(t)$, there exists $\alpha \in \Theta$ with $\|\alpha -
\theta\|_2 \leq t$ such that
\begin{equation*}
  A_{\theta}(t) \leq \left<Z, \alpha - \theta \right> - f(\alpha) +
  \eta - \frac{t^2}{2}.
\end{equation*}
Because $\|\alpha - \theta\|_2 \leq t$, we can write
\begin{align*}
  A_{\theta}(t) &\leq \left<Z, \alpha - \theta \right> - f(\alpha) +
  \eta - \frac{t^2}{2} \\
&\leq \left<Z, \alpha - \theta \right> - f(\alpha) +
  \eta - \frac{\|\alpha  - \theta\|_2^2}{2} \\
&\leq \tilde{A}_{\theta}(\|\alpha - \theta\|_2) + \eta \\
&\leq \tilde{A}_{\theta}(\|\hat{\theta} - \theta\|_2) + \eta
\qt{because $\|\hat{\theta} - \theta\|_2$ maximizes
  $\tilde{A}_{\theta}(\cdot)$} \\
&\leq A_{\theta}(\|\hat{\theta} - \theta\|_2) + \eta \qt{because
  $\tilde{A}_{\theta}(t) \leq A_{\theta}(t)$ for every $t \geq 0$}.
\end{align*}
Because $t \geq 0$ and $\eta > 0$ are arbitrary, we have proved that
$\|\hat{\theta} - \theta\|_2$ maximizes $A_{\theta}(t)$ over $t \geq
0$.

Note now that $A_{\theta}(t)$ is a concave
function of $t \geq 0$. This is because the function
\begin{equation*}
  t \mapsto \sup_{\alpha \in \Theta: \|\alpha - \theta\|_2 \leq t}
  \left(\left<Z, \alpha - \theta \right> - f(\alpha) \right)
\end{equation*}
is concave as shown in the proof of Lemma \ref{ttm}(1) and also
$t \mapsto -t^2/2$ is trivially concave. As a result of the concavity
of $A_{\theta}(\cdot)$, it follows that for every $\delta \geq 0$,
\begin{equation}\label{hp1}
  \|\hat{\theta} - \theta\|_2 < t_{\theta} + \delta
\end{equation}
provided
\begin{eqnarray}
&&  \max \left(\E_{\theta} A_{\theta}(t_{\theta}) - A_{\theta}(t_{\theta}),
    A_{\theta}(t_{\theta} + \delta) - \E_{\theta} A_{\theta}(t_{\theta} +
    \delta) \right) \nonumber \\
    &< & \frac{1}{2} \left\{\E_{\theta} A_{\theta}(t_{\theta})
  - \E_{\theta} A_{\theta}(t_{\theta} + \delta)\right\}.  \label{hp2}
\end{eqnarray}
To see this, assume that \eqref{hp2} holds for some $\delta \geq
0$. Then, if $B := (\E_{\theta} A_{\theta}(t_{\theta}) + \E_{\theta}
A_{\theta}(t_{\theta} + \delta))/2$, then \eqref{hp2} implies that
\begin{equation*}
  A_{\theta}(t_{\theta}) > \E_{\theta} A_{\theta}(t_{\theta}) - \frac{1}{2}
  \left(\E_{\theta} A_{\theta}(t_{\theta}) - \E_{\theta} A_{\theta}(t_{\theta} + \delta)
  \right) = B
\end{equation*}
and also
\begin{equation*}
  A_{\theta}(t_{\theta} + \delta) < \E_{\theta} A_{\theta}(t_{\theta} + \delta)
  + \frac{1}{2} \left(\E_{\theta} A_{\theta}(t_{\theta}) - \E_{\theta}
    A_{\theta}(t_{\theta} + \delta) \right) = B.
\end{equation*}
We thus have $A_{\theta}(t_{\theta} + \delta) < B <
A_{\theta}(t_{\theta})$. Because $A_{\theta}(\cdot)$ is concave, this
implies that every maximizer of $A_{\theta}$ has to be strictly
smaller than $t_{\theta} + \delta$ which proves \eqref{hp1}. From
\eqref{hp1}, we immediately have
\begin{eqnarray}\label{namo}
&&  \P_{\theta}  \left\{\|\hat{\theta} - \theta\|_2 \ge t_{\theta} + \delta
  \right\}   \\
  &\leq &  \P_{\theta} \left\{\E_{\theta} A_{\theta}(t_{\theta}) -
    A_{\theta}(t_{\theta}) \geq \frac{\Delta}{2} \right\} + \P_{\theta} \left\{
    A_{\theta}(t_{\theta} + \delta) - \E_{\theta} A_{\theta}(t_{\theta} +
    \delta)
    \geq \frac{\Delta}{2}\right\} \nonumber
\end{eqnarray}
where
\begin{equation*}
  \Delta := \E_{\theta} A_{\theta}(t_{\theta}) - \E_{\theta}
  A_{\theta}(t_{\theta} +
  \delta).
\end{equation*}
We now note that $\E_{\theta} A_{\theta}(t) = m_{\theta}(t) - t^2/2 =
G_{\theta}(t)$ where $G_{\theta}(t)$ and $m_{\theta}(t)$ are defined
in \eqref{gmd}. Therefore, from Lemma \ref{ttm}(4), we get
\begin{equation*}
  \Delta = G_{\theta}(t_{\theta}) - G_{\theta}(t_{\theta} + \delta)
  \geq \frac{\delta^2}{2}.
\end{equation*}
Thus \eqref{namo} gives
\begin{eqnarray*}
&&  \P_{\theta} \left\{\|\hat{\theta} - \theta\|_2 \ge t_{\theta} + \delta
  \right\} \\
  &\leq &  \P_{\theta} \left\{\E_{\theta} A_{\theta}(t_{\theta}) -
    A_{\theta}(t_{\theta}) \geq \frac{\delta^2}{4} \right\} + \P \left\{
    A_{\theta}(t_{\theta} + \delta) - \E_{\theta} A_{\theta}(t_{\theta} +
    \delta)
    \geq \frac{\delta^2}{4}\right\}
\end{eqnarray*}
We now use the trivial fact that for every $t \geq 0$, the quantity
$A_{\theta}(t)$, as a function of $Z$, is Lipschitz with Lipschitz
constant $t$. Therefore, by standard concentration for Lipschitz
functions of Gaussian random variables, we get
\begin{equation*}
  \P_{\theta} \left\{\|\hat{\theta} - \theta\|_2 \ge t_{\theta} + \delta
  \right\} \leq \exp \left(\frac{-\delta^4}{32 t_{\theta}^2} \right) +
  \exp \left(\frac{-\delta^4}{32 (t_{\theta} + \delta)^2} \right) \leq
  2 \exp \left(\frac{-\delta^4}{32 (t_{\theta} + \delta)^2} \right)
\end{equation*}
which proves inequality \eqref{erc.3}.

   We now turn to the proof of \eqref{1co}. For convenience, let $L :=
   \|\widehat{\theta}(X; \Theta, f) - \theta\|_2$. First assume that
   $t_{\theta} \geq 1$. Using \eqref{erc.3}, we can write
   \begin{align*}
     \P_{\theta} \left\{L \geq
       t_{\theta} + x \sqrt{t_{\theta}} \right\} \leq 2 \exp
     \left(\frac{-x^4}{32 \left(1 + x t^{-1/2}_{\theta} \right)^2}
     \right)  \leq 2 \exp \left(\frac{-x^4}{32(1 + x)^2} \right).
   \end{align*}
   As a result, via the identity $\E X_+^2 = 2 \int_0^{\infty} x \P
   \{X \geq x\} dx$ which holds for every random variable $X$, we
   obtain
   \begin{align*}
    \E_{\theta} \left(\frac{L - t_{\theta}}{\sqrt{t_{\theta}}} \right)^2_+ &=
     2 \int_0^{\infty} x ~ \P_{\theta} \left\{L \geq t_{\theta} + x \sqrt{t_{\theta}} \right\} dx
     \\
&\leq 4 \int_0^{\infty} x ~ \exp \left(\frac{-x^4}{32(1 + x)^2}
\right) dx \leq 84.
   \end{align*}
  where we have also used the fact that the integral above is at most
  $21$ (as can be verified by numerical computation). Note that the
  above bound also implies that
  \begin{equation*}
    \E_{\theta} \left(\frac{L - t_{\theta}}{\sqrt{t_{\theta}}} \right)_+ \leq
     \sqrt{84}.
  \end{equation*}
  Thus if $L := \|\widehat{\theta}(X; \Theta, f) - \theta\|_2$, then
  the inequality
  \begin{equation}\label{yako}
    L \leq (L - t_{\theta})_+^2 + t_{\theta}^2 + 2 t_{\theta} (L - t_{\theta})_+
  \end{equation}
together with the above two bounds for $\E (L - t_{\theta})_+^2$ and
$\E_{\theta} (L - t_{\theta})_+$ proves \eqref{1co} in the case when
$t_{\theta} \geq 1$.

 When $t_{\theta} \leq 1$, inequality \eqref{erc.3} gives
 \begin{equation*}
   \P_{\theta} \left\{L \geq t_{\theta} + x\right\} \leq 2 \exp
   \left(\frac{-x^4}{32(t + x)^2} \right) \leq 2 \exp
   \left(\frac{-x^4}{32(1 + x)^2} \right)
 \end{equation*}
  which implies, as before, that
  \begin{equation*}
    \E_{\theta} (L - t_{\theta})_+^2 \leq 84 ~~ \text{ and } ~~ \E_{\theta} (L -
    t_{\theta})_+ \leq \sqrt{84}.
  \end{equation*}
 so that the required inequality \eqref{1co} again follows from
 \eqref{yako}.
\end{proof}

\section*{Acknowledgment}
We are thankful to Sivaraman Balakrishnan for helpful discussions.

\bibliographystyle{chicago}
\bibliography{AG}

\end{document}